\documentclass[12pt]{amsart}

\usepackage{amsthm}
\usepackage{amssymb}

\usepackage{enumerate}

\makeatletter
\@namedef{subjclassname@2010}{%
  \textup{2010} Mathematics Subject Classification}
\makeatother

\theoremstyle{plain}
\newtheorem{theorem}{Theorem}[section]
\newtheorem{lemma}[theorem]{Lemma}
\newtheorem{proposition}[theorem]{Proposition}
\newtheorem{corollary}[theorem]{Corollary}

\theoremstyle{definition}
\newtheorem{remark}[theorem]{Remark}

\newcommand{\C}{\mathbb C}

\newcommand{\Q}{\mathbb Q}
\newcommand{\R}{\mathbb R}
\newcommand{\N}{\mathbb N}

\newcommand{\diff}{\mathop{}\!\mathrm{d}}

\newcommand{\abs}[1]{\left| #1 \right|}

\DeclareMathOperator{\re}{Re}

\DeclareMathOperator{\intr}{int}
\DeclareMathOperator{\Gr}{Gr}
\DeclareMathOperator{\hull}{-hull}
\DeclareMathOperator{\unif}{unif}

\begin{document}


\baselineskip=17pt


\title[Random  polynomials and random rational functions]{Approximation by random complex polynomials and random rational functions}

\dedicatory{Dedicated to the memory of J.  Siciak}

\author[Gauthier]{Paul M. Gauthier}
\address{D\'epartement de math\'ematiques et de statistique,  Universit\'e de Mont\-r\'eal,
CP-6128 Centreville, Montr\'eal,  H3C3J7, CANADA}
\email{gauthier@dms.umontreal.ca}
 
\author[Ransford]{Thomas Ransford}
\address{D\'epartement de math\'ematiques et de statistique, Universit\'e Laval, Qu\'e\-bec City,
G1V0A6, CANADA}
\email{thomas.ransford@mat.ulaval.ca}

\author[St-Amant]{Simon St-Amant}
\address{D\'epartement de math\'ematiques et de statistique, Universit\'e de Mont\-r\'eal,
CP-6128 Centreville, Montr\'eal,  H3C3J7, CANADA}
\email{simon.st-amant@umontreal.ca}

\author[Turcotte]{J\'er\'emie Turcotte}
\address{D\'epartement de math\'ematiques et de statistique, Universit\'e de Mont\-r\'eal,
CP-6128 Centreville, Montr\'eal,  H3C3J7, CANADA}
\email{jeremie.turcotte@umontreal.ca}

\date{19 February 2019}

\begin{abstract}
We investigate random compact sets with random functions defined thereon, such as polynomials, rational functions, the pluricomplex Green function and the Siciak extremal function. One surprising consequence of our study is that randomness can be used to `improve' convergence for sequences of functions.
\end{abstract}



\maketitle

\setcounter{section}{-1}
\section{Introduction}

We investigate functions and sets with a measurable parameter. For nearly five decades these concepts have been studied in many different mathematical contexts under the names  random functions and random sets, even if sometimes the actual randomness played no particular role. 

The main objective of this article is to generalize complex approximation theorems to the context of random functions. 

Our main results are the following :
\begin{itemize}
    \item A generalization of Runge's Theorem (Theorem \ref{Runge});
    \item A generalization of the Oka--Weil Theorem (Theorem \ref{OkaWeil});
    \item The image of a random function over a compact set is a random compact set (Theorem \ref{F sub f});
    \item The polynomially and rationally convex hulls of a random compact set are random compact sets (Theorems \ref{P-hull measurable} and \ref{R-hull measurable});
    \item The Siciak extremal function and the pluricomplex Green function of a random  compact set are random functions (Theorem \ref{siciak});
    \item A  useful convergence theorem, that states that from a weak form of convergence we may extract stronger convergence (Theorem \ref{convUnif}).
\end{itemize}

We highlight in more detail three results that illustrate the kind of theorems that we are going to establish.
Firstly, a random Runge theorem.

\medskip
{\bf Theorem \ref{conjRunge}.} {\em
Let $K\subset\C$ be an arbitrary non-empty compact set, $(\Omega, \mathcal A)$ a measurable space and 
$f:\Omega\times K\rightarrow\C$ a random function on $K.$  Suppose that,  for each $\omega\in \Omega,$  $f(\omega,\cdot)$ is the restriction to some open neighborhood $U_\omega$ of $K$ of some function $g_\omega$ holomorphic on $U_\omega.$ 

Then, there is a sequence $r_j(\omega,z)$ of random rational functions, pole-free on  $K$, such that
$$
\mbox{for each} \,\, \omega, \quad r_j(\omega,\cdot)\rightarrow f(\omega,\cdot) \quad \mbox{uniformly on} \quad K.
$$   
}

The second result, a striking example of the power of the selection method, came as a surprise to us.  It shows that the approximation in Theorem \ref{conjRunge} implies that a stronger approximation is possible.

\medskip
{\bf Corollary \ref{surprise}.} {\em
Let $K\subset\C^n$ be an arbitrary non-empty compact set, $(\Omega, \mathcal A)$ a measurable space and 
$f:\Omega\times K\rightarrow\C$  a mapping. The following are equivalent. 

1. There is a sequence $r_j(\omega,z)$ of random rational functions, pole-free on  $K$, such that,  
$$
\mbox{for each} \,\, \omega, \quad r_j(\omega,\cdot)\rightarrow f(\omega,\cdot) \quad \mbox{uniformly on} \quad K.
$$

2. There is a sequence $r_j(\omega,z)$ of random rational functions, pole-free on  $K$, such that,  
$$
	r_j(\omega,z)\rightarrow f(\omega,z) \quad \mbox{uniformly on} 
		\quad \Omega\times K.
$$
}

The third result we wish to highlight generalizes the Oka-Weil Theorem.

\medskip
{\bf Theorem \ref{OkaWeil}.} {\em
Let $(\Omega,\mathcal A, \mu)$ be a $\sigma$-finite measure space. Let $K$ be a random compact set, whose range $\{K(\omega):\omega\in\Omega\}$ consists of at most a countable number of different compact sets. Suppose that $K$ is polynomially convex, i.e., $K(\omega)=\widehat K(\omega)$ for all~$\omega$. Let $f$ be a generalized random function such that $f(\omega,\cdot)$ is holomorphic in a neighborhood of $K(\omega)$ for each $\omega$, and let $\varepsilon$ be a positive measurable function defined on $\Omega.$ Then there exists $p$, a generalized random polynomial, such that$$\|p(\omega,\cdot)-f(\omega,\cdot)\|_{K(\omega)}<\varepsilon(\omega)$$
	for all $\omega$ outside a  measurable  set $L\subset \Omega$ such that $\mu(L)=0$.
}

\medskip
The terms used in this theorem will be explained when they first arise.

If $(\Omega, \mathcal A)$ and $(Z, \mathcal B)$ are two measurable spaces and $X$ is an arbitrary non-empty set,
we shall say that a function $f:\Omega\times X\rightarrow Z$ is  a 
random  function
on $X,$ if the function $f(\cdot,x)$ is measurable for each $x\in X.$ 
Clearly, if  $f:\Omega\times X\rightarrow Z$ is a  random function and $Y\subset X,$ then the restricted mapping  $f:\Omega\times Y\rightarrow Z$ is a random function. 
Suppose $f:\Omega\times X\rightarrow Z$ is a  random function on $X$ and both $X$ and $Z$ are also equipped with topologies. If $f(\omega,\cdot)$ is a continuous function on $X$ for each $\omega\in\Omega$, 
then we shall say that $f$  is a random continuous function on $X.$  Similarly, we shall speak of  random polynomials, random rational functions, random holomorphic functions  etc.

We remark that  a random function need not be jointly measurable. We thank Eduardo Zeron for bringing to our attention an  example  given by Sierpinski (see \cite[p.167]{Ru1987}). However, a function that is measurable in one variable and continuous in the other 
(such functions are called Carath\'eodory functions) is jointly measurable  (see \cite[Theorem 3.1.30]{S}).  
The same result holds for complex-valued functions.

For other studies on the topic of this paper and for historical background, see  \cite{AF}, \cite{BS1989}, \cite{BS1996}, \cite{H1975}, \cite{HV1969}, \cite{No},    and some more literature found in these references. We thank the referees for helpful comments.


\section{Measurable functions}

We  use the Borel $\sigma$-algebra on $\R^n$ (and $\C^n$  viewed as $\R^{2n}$) to define measurability. To facilitate the reading of this paper, we collect here some well known basic facts that will be used  throughout the paper.

\begin{proposition} \label{measurableInt}
Let $(\Omega,\mathcal A)$ be a measurable space and $Q$ be a hypercube in $\R^n.$ Suppose $f : \Omega \times Q \rightarrow\C$ is such that $f(\cdot, x)$ is measurable for all $x$ and $f(\omega, \cdot)$ is Riemann integrable on  $Q$ for all $\omega \in \Omega$. Then the function $F(\omega): = \int_Q f(\omega,x) \diff x$ is measurable.
\end{proposition}

\begin{proof}
Since $f(\omega,\cdot)$ is Riemann integrable for each $\omega,$ we have pointwise convergence of the Riemann sums to $F$. We easily conclude that $F$ is measurable.
\end{proof}

Let $K$ be a compact metric space and let  $C(K)$ be the Banach algebra of continuous complex-valued functions on $K$ equipped with the sup norm.  
We consider $C(K)$ as a measurable space endowed with 
the $\sigma$-algebra ${\mathcal B}$ of Borel subsets of $C(K)$.

\begin{proposition}\label{measurable element} 
Let $F:\Omega\rightarrow C(K)$ be a  mapping, and define a function
$$
	f:\Omega\times K\rightarrow \C, \quad (\omega,x)\mapsto F(\omega)(x).
$$
Then $F$ is measurable if and only if 
$f(\cdot, x)$ is measurable for all $x\in K$. 
\end{proposition}

\begin{proof}
Let $F$ be measurable.
For  $x$ fixed, 
$$f(\omega,x)=F(\omega)(x)=\phi_x(F(\omega))=(\phi_x\circ F)(\omega),$$ where $\phi_x:C(K)\rightarrow \C$ is the evaluation functional at the point $x$.
Thus $f(\cdot,x)$ is the composition of the measurable function $F$ with the continuous function $\phi_x$. Therefore $f(\cdot,x)$ is indeed measurable. 

We write $\overline{B_r}(g)=\{j \in C(K) : \sup_{x\in K} |j(x)-g(x)|\leq r\}$, the closed ball of radius $r>0$ around a point $g$ of $C(K)$. Also, we write $f_x^{-1}(E)=\{\omega \in \Omega : f(\omega, x)\in E\}$.

For the converse, since $K$ is compact metric, $C(K)$ is separable. Thus it suffices to show that $F^{-1}\overline{B_r(g)}\in\mathcal A.$ By hypothesis, $K$ has a countable dense set $\{x_j\}.$ Let $B_r=\{z\in\mathbb C:|z|\le r\},$ and for $x\in K$ let  
$$
	f_x(\omega)=|g(x)-F(\omega)(x)|=|g(x)-f(\omega,x)|.
$$
Then each $f_x$ is measurable and 
$$
	F^{-1}\overline{B_r(g)} = \bigcap_{j=1}^\infty f_{x_j}^{-1}(B_r) \in \mathcal A.
$$

\end{proof}

\begin{corollary}\label{maps F to f}
Let $K$ be a compact metric space,  let
$F:\Omega\rightarrow C(K,\C^n)$ be a mapping, and define a function
$$
	f:\Omega\times K\rightarrow\C^n, \quad (\omega,x)\mapsto F(\omega)(x).
$$
Then $F$ is measurable if and only if $f(\cdot, x)$ is measurable for each $x\in K$.
\end{corollary}

\begin{proof}
For $j=1,\ldots,n,$ 
let $F_j(\omega)$ be the components of $F(\omega)$ and, define functions 
$$
	f_j:\Omega\times K\rightarrow\C, \quad (\omega,x)\mapsto F_j(\omega)(x).
$$
By Proposition \ref{measurable element}, for each $j=1,\ldots n,$ 
the map $F_j$ is measurable if and only if $f_j(\cdot,x)$ is measurable for all $x\in K$. The result follows.
\end{proof}


\section{Runge's theorem for random functions}

If $U$ is an open set in $\C,$ we say that a function $f:\Omega\times U\rightarrow \C$ is a random holomorphic function on $U$ if $f(\omega,\cdot)$ is holomorphic for each  
$\omega\in \Omega$ 
and $f(\cdot,z)$ is measurable for each $z\in U.$ 
We define a random polynomial $p$  as a random function  $p:\Omega\times\C\rightarrow\C,$ such that $p(\omega,\cdot)$ is a polynomial for each $\omega\in\Omega.$ 

Consider the following example. Let $\Omega = \Omega_0 \cup \Omega_1\cup\Omega_2\cup\cdots$ be a partition of $\Omega$ into non-empty measurable subsets and, for $n=0,1,2,\ldots$,  let $a_n$ be the characteristic function of $\Omega_n.$ Then,
$$
	p(\omega,z) = \sum_{n=0}^\infty a_n(\omega)z^n
$$
is a random polynomial.  
Indeed, for fixed $\omega,$ 
the function $p(\omega,z)$ is the polynomial $z^n,$ where $n$ is the unique natural number for which $\omega\in\Omega_n.$ For fixed $z,$ the function $p(\omega,z)$ is measurable, since it is the pointwise limit of the partial sums and the latter are measurable.
Notice that $p$ has infinite degree, although for $\omega\in\Omega_n,$ $p$ has degree $n,$ since $p(\omega,z)=z^n.$

For the extended complex plane $\C\cup\{\infty\},$ we use the notation $\overline\C.$ 
If $f:\Omega\times X\rightarrow\C$ is a random  complex function on some set $X,$ then we may also consider $f$ as a random function  $f:\Omega\times X\rightarrow\overline\C$ since, for every Borel set $B\subset\overline\C$ and every $x\in X,$ 
$$
	\{\omega:f(\omega,x)\in B\} = \{\omega:f(\omega,x)\in B\cap\C\},
$$ 
is measurable, and $B\cap\C$ is a Borel subset of $\C.$ In particular, every random  polynomial, which by definition is a $\C$-valued random function,  can be considered as a  $\overline\C$-valued random function.

For a compact set $K\subset\C,$ let us define a random rational function  pole-free on $K$ as a random rational function $f:\Omega\times K\rightarrow\C,$ such that, for each $\omega\in \Omega,$  the function $f(\omega,\cdot)$ is a rational function pole-free on  $K.$  

Let $K$ be a compact set in $\C^n$ and $f$ be a random
function on $K$  that is continuous on $K$ and holomorphic on the interior of $K$. We say that $f$ is in $R_\Omega(K)$ if there exists a sequence $r_j$ of random rational functions 
pole-free on $K$  such that, 
for every $\omega \in \Omega,$
 $r_j(\omega, \cdot) \rightarrow f(\omega, \cdot)$ uniformly on $K$.
Similarly, we say that $f$ is in $R_\Omega^{\unif}(K)$ if, for every $\varepsilon > 0$ there exists a random rational function $r$ pole-free on $K$ 
such that $\abs{r(\omega, z) - f(\omega, z)} < \varepsilon$ for all $(\omega, z)$ in $\Omega \times K$.

It would  be useful to have sufficient conditions  for a random function $f$  to be in  $R_\Omega(K).$ Let us call such a result a Runge theorem for random functions.  
Andrus and Brown \cite{AB1984} obtained such a  Runge theorem, in which the compact set $K$ was also random. 
We now formulate our first version of a Runge theorem, in which the compact set $K$ is classic (parameter free) and the approximation is everywhere. 
Another difference between our presentation and that in \cite{AB1984} is that we have given an explicit definition of random  ``rational function 
pole-free on  $K$,'' whereas in \cite{AB1984} no explicit definition of (the corresponding notion) random rational function is given.

\begin{theorem}\label{Runge 1}
Let $U$ be an open set in $\C$ and $f:\Omega\times U \rightarrow \C$ be a random holomorphic function. Let $K$ be a compact subset of $U$. 
Then there exists a sequence $R_1, R_2, \ldots,$ of random rational functions  pole-free on 
$K$ such that,  for each
$\omega\in\Omega,$ 
$$
	R_n(\omega,\cdot)\rightarrow f(\omega,\cdot) \quad \mbox{uniformly on} \quad K. 
$$
\end{theorem}

\begin{proof}
We can cover $K$ by finitely many disjoint compact sets $Q_1,\cdots,Q_n,$ such that each $Q_k$ is contained in $U,$  each $Q_k$ is bounded by finitely many disjoint polygonal curves and $K\subset \cup_k \intr Q_k.$ 
Let  $\Gamma = \cup_k\partial Q_k.$  
By the Cauchy formula, for each $\omega\in\Omega,$ 
$$
	f(\omega, z) = \frac{1}{2\pi i} \int_\Gamma \frac{f(\omega, \zeta)}{\zeta - z}\diff\zeta		
		\quad,\quad  \forall z\in K.
$$
For $\delta>0,$ partition $\Gamma$ into $N=N(\delta)$ segments $\Gamma_j$ of length smaller than $\delta$.  For each $\Gamma_j,$ denote by $\zeta_j$ the terminal point of $\Gamma_j.$ 
The Riemann sum 
\begin{equation*}
R(\omega, z) = \sum_{j=1}^{N(\delta)} \frac{1}{2\pi i} \frac{f(\omega, \zeta_j)}{\zeta_j - z} \int_{\Gamma_j} \diff\zeta =  \sum_{j=1}^{N(\delta)} \frac{a_j(\omega)}{\zeta_j-z}
\end{equation*}
is a random rational function pole-free on  $K.$ Put
$$
	 \eta(\omega,\delta) := \max 
\left\{\frac{1}{2\pi}\left|\frac{f(\omega, \zeta)}{\zeta - z} - \frac{f(\omega, w)}{w - z}\right|:
	\zeta, w \in \Gamma, |\zeta-w|<\delta,  z\in K\right\}. 
$$
For all $(\omega,z)\in\Omega\times K,$ 
$$
	\abs{f(\omega,z) - R(\omega,z)} < \eta(\omega,\delta)\cdot L(\Gamma),
$$
where $L(\Gamma)$ is the length of $\Gamma.$ It follows from the uniform continuity of $(\zeta,z)\mapsto f(\omega,\zeta)/(\zeta-z)$ on $\Gamma\times K$ that, if $\delta=\delta(\omega)$ is sufficiently small, then $\eta(\omega,\delta)<\varepsilon/L(\Gamma).$ Thus, 
$$
	\abs{f(\omega,z) - R(\omega,z)} < \varepsilon, \quad \forall z\in K. 
$$

Let $\{\delta_n\}_n$ be a sequence of positive numbers decreasing to zero and, for each $\delta_n,$  let $R_n$ be a random rational function pole-free on  $K,$
corresponding as above to $\delta_n.$ Then, for each
$\omega\in\Omega,$ 
$$
	R_n(\omega,\cdot)\rightarrow f(\omega,\cdot) \quad \mbox{uniformly on} \quad K. 
$$
\end{proof}

Theorem \ref{Runge 1} allows us to approximate a random  function $f$ on a compact set $K,$ provided there is an open neighborhood $U$ of $K,$ such $f(\omega, \cdot)$ is holomorphic on 
$U$ for all $\omega \in \Omega.$ This condition 
is quite strong and 
we shall now set the stage for  a better version of Runge's theorem. 
 
For an open subset $V\subset\C,$ denote by $C^1(V)$ the family of continuously differentiable functions on $V$ and for $g\in C^1(V),$ and $M\ge 0,$ put
$$
	\|g\|_1:=\sup\left\{\max\{|g(z)|,\|\nabla g(z)\|\}: z\in V\right\},
$$
$$
	C^1(V,M):= \{g\in C^1(V): \|g\|_1\le M\}.	
$$ 

Let $\Gamma$ be a finite union of disjoint smooth curves in $V.$  For a partition $\mathcal P=(\zeta_1,\zeta_2,\ldots,\zeta_\ell)$ of $\Gamma$ and a function $g\in C^1(V),$  we denote by 
$\sum_{g,\mathcal P}$ the Riemann sum 
$$
	\sum_{g,\mathcal P} := \sum_{\gamma_j}g(\zeta_j)\int_{\gamma_j} d\zeta,
$$
where $\Gamma:=\sum_j\gamma_j$ is the decomposition of $\Gamma$ into arcs induced by the partition $\mathcal P$ and $\zeta_j$ is the initial point of the arc $\gamma_j.$ 

\begin{lemma}\label{Runge lemma} Let $V, \Gamma, M$ be  as above. Then, for  each $\varepsilon>0,$ there exists a partition $\mathcal P$ such that, for each $g\in C^1(V,M)$,
$$
	\left|\int_\Gamma g(\zeta) - \sum_{g,\mathcal P}\right| < \varepsilon.
$$
\end{lemma}

\begin{theorem}\label{Runge 2} Let $U$ be an open subset of $\C$ and $K$  a compact subset of $U.$ Let $f:\Omega\times U\rightarrow\C$ be a random function on $U$ such that, for each $\omega,$ there is an open neighborhood $U_\omega$ of $K$ in $U$ for which the restriction of $f(\omega,\cdot)$ to $U_\omega$ is holomorphic. Then there is a sequence $\{R_k(\omega,z)\}$ of random rational functions pole-free on  $K,$
such that, for each $\omega,$ $R_k(\omega,\cdot)\rightarrow f(\omega,\cdot)$ uniformly on $K.$
\end{theorem}

\begin{proof} 
We may assume that $U_\omega$ is relatively compact in $U$ and that $f(\omega,\cdot)$ is holomorphic on a neighborhood of $\overline U_\omega.$ Thus $f(\omega,\cdot)\in C^1(U_\omega, M_\omega),$ for some finite $M_\omega.$

For  $(\omega,\zeta,z)\in\Omega\times U\times \C,$ 
put 
$$
	g_{\omega,z}(\zeta)=\frac{1}{2\pi i}\frac{f(\omega,\zeta)}{\zeta-z}.
$$
Take $\Gamma_\omega$ in $U_\omega$ such that $\Gamma_\omega$ has index  $1$ with respect to each $z\in K$ and index $0$ with respect to each $z\in \C\setminus U_\omega$. Then, by the Cauchy formula,
$$
f(\omega,z) = \int_{\Gamma_\omega}g_{\omega,z}(\zeta)d\zeta, \quad 
		\mbox{for} \quad z\in K.
$$

Let $U_k, k=1,2,\ldots, $ be a neighborhood basis of $K.$ 
For each $k,$ choose a $\Gamma_k\subset U_k\setminus K$ such that $\Gamma_k$ has index 1 with respect to each point of $K$
and index $0$ with respect to each $z\in\C\setminus U_k$.  Let $V_k$ be a bounded open neighborhood of $\Gamma_k$ in $U_k,$ whose closure is disjoint from $K.$ For each $k,$ there is a finite $\mu_k$ such that $1/(\zeta-z)$ as a function of $\zeta$ is in $C^{1}(V_k,\mu_k),$ for each $z\in K.$ 
For each $\omega,$ we have $f(\omega,\cdot)\in C^{1}(U_\omega,M_\omega),$ so $f(\omega,\cdot)\in C^{1}(U_k,M_\omega)$ for all $k$ such that $U_k\subset U_\omega.$ Hence 
there is a $k(\omega)$ such that, for all $k>k(\omega),$ we have that $U_k\subset U_\omega$ and  
$$
	g_{\omega,z}\in C^{1}(V_k,k\mu_k), \quad \forall  k\ge k(\omega),
		\quad \forall z\in K.
$$

In Lemma \ref{Runge lemma}, we replace $g$ by $g_{\omega,z},$ $V$ by $V_k,$ $\Gamma$ by $\Gamma_k,$ $M$ by $k\mu_k$ and $\varepsilon$ by $1/k.$ 
In the Riemann sums, we have the terms 
$$
	g_{\omega,z}(\zeta_j) = \frac{1}{2\pi i}\frac{f(\omega,\zeta_j)}{\zeta_j-z},
$$
and, by hypothesis, each $f(\cdot,\zeta_j)$ is measurable. Hence, 
the Riemann sums of $g_{\omega,z}$ are random rational functions pole-free on  $K,$
which by the lemma perform the appropriate approximation. 
This concludes the proof of the theorem. 
\end{proof}

Even Theorem~\ref{Runge 2} has a stronger hypothesis than necessary.
The following result shows that we can drop the open set $U$ in the statement of Theorem~\ref{Runge 2}.

\begin{theorem} \label{conjRunge} 
Let $K\subset\C$ be an arbitrary non-empty compact set, $(\Omega, \mathcal A)$ a measurable space and 
$f:\Omega\times K\rightarrow\C$ a random function. Suppose that,  for each $\omega\in \Omega,$  $f(\omega,\cdot)$ is the restriction to some open neighborhood $U_\omega$ of $K$ of some function $g_\omega$ holomorphic on $U_\omega.$ 
Then there is a sequence $r_j(\omega,z)$ of random rational functions pole-free on  $K$, such that,  
$$
\mbox{for each} \,\, \omega, \quad r_j(\omega,\cdot)\rightarrow f(\omega,\cdot) \quad \mbox{uniformly on} \quad K.
$$   
\end{theorem}

We shall establish this result in the next section.
In Section \ref{GeneralizedmeasurableFunctions}, 
we shall prove one of our main results stating that, under appropriate hypotheses, separately uniform convergence implies joint uniform convergence.

\section{A measurable extension theorem}

In view of Theorem \ref{Runge 2}, to establish Theorem \ref{conjRunge}, it suffices to prove the following measurable extension theorem.

\begin{theorem}\label{T:extension}
Let $K$ be a compact subset of $\C,$ let $(\Omega,\mathcal A)$
be a measurable space, and let $f : \Omega\times K \rightarrow \C$
be a function such that:
\begin{enumerate}[(i)$'$]
\item $f(\cdot,z)$ is measurable for each $z\in K$,
\item for each $\omega\in\Omega,$ there exist an open neighborhood $U_\omega$ of $K$ and a function $g_\omega$ holomorphic on $U_\omega$ such that $g_\omega|_K = f(\omega,\cdot)$.
\end{enumerate}
Then there exists a function $F : \Omega\times \C\rightarrow \C$ such that:
\begin{enumerate}[(i)]
\item $F(\cdot,z)$ is measurable for each $z \in\C$,
\item  for each $\omega\in\Omega,$ there exist an open neighborhood $U_\omega$ of $K$ and a function $g_\omega$ holomorphic on $U_\omega$ such that $g_\omega|_K = F(\omega,\cdot)$,
\item $F|_{\Omega\times K} = f.$
\end{enumerate}
\end{theorem}

The proof of Theorem~\ref{T:extension} will occupy the rest of this section. 

Given a Banach space $X,$ we write $X^*$ for the dual space of $X$.
The following result is well known.

\begin{lemma}\label{L:convex} Let $X$ be a separable Banach space and let $C$ be a closed convex subset of $X.$ Then
there exist sequences $(\phi_n)_{n\ge 1} \in X^*$ and $(\alpha_n)_{n\ge 1} \in\R$ such that
\begin{equation}\label{C}
	C = \bigcap_{n\ge 1}\{x\in X: \re\phi_n(x)\le\alpha_n\}.
\end{equation}
\end{lemma}

\begin{proof}
As $X$ is separable, there is a countable base for its topology consisting of open balls. Let
$(B_n)$ be an enumeration of the set of these balls having the property that $B_n\cap C = \emptyset.$ By the
Hahn-Banach theorem \cite[Theorem 3.4(a)]{Ru1991}, there exist $\phi_n\in X^*$ and 
$\alpha_n\in\R,$ such that
$$ 
	\re\phi_n(x)\le \alpha_n \quad (x \in C) \quad \mbox{and} 
		\quad   \re\phi_n(x)>\alpha_n  \quad (x \in B_n).
$$
Then (\ref{C}) holds.
\end{proof}

Given an open subset $U$ of $\C,$ we denote by $A^2(U)$ the Bergman space on $U$, namely the subspace of $L^2(U)$ consisting of functions holomorphic on $U$.
It is well known that $A^2(U)$ is a closed subspace of $L^2(U),$ and therefore a Hilbert space. Also, convergence of a sequence in $A^2(U)$ implies uniform convergence on each compact subset of $U.$

\begin{lemma}\label{Bergman}
Let $K$ be a compact subset of $\C,$ let $(\Omega,\mathcal A)$ be a measurable space, and let $f:\Omega\times K\rightarrow\C$  be a function such that:
\begin{itemize}
\item $f(\cdot,z)$ is measurable for all $z\in K$,
\item $f(\omega,\cdot)\in C(K)$ for all $\omega\in\Omega$.
\end{itemize}
Then, for each open neighborhood $U of K,$ we have $\{\omega:f(\omega,\cdot)\in A^2(U)|_K\}\in\mathcal A.$
\end{lemma}

\begin{proof}
Fix $U.$ The restriction map $R_K : A^2(U)\rightarrow C(K)$ is linear and continuous, so it is also continuous with respect to the weak topologies on $A^2(U)$ and $C(K).$ The  closed  unit ball $B$ of $A^2(U)$ is
weakly compact, because $A^2(U)$ is a Hilbert space. Therefore $C := R_K(B)$ is weakly compact in $C(K).$ Hence $C$ is  closed in $C(K).$ Clearly it is also convex. By Lemma~\ref{L:convex}, there are sequences
$(\phi_n)\in C(K)^*$ and $(\alpha_n)\in\R$ such that $C = \cap_n\{g\in C(K) : \re \phi_n(g)\le \alpha_n\}.$ It follows that
$$
	\{\omega\in\Omega: f(\omega,\cdot)\in C\} =
		\bigcap_{n\ge 1}\{\omega\in\Omega:\re \phi_n(f(\omega,\cdot))\le\alpha_n\}.
$$
By Proposition~\ref{measurable element}  each of the functions $\omega\mapsto\phi_n(f(\omega,\cdot))$ is measurable. Therefore
$$
	\{\omega: f(\omega,\cdot)\in C\}\in \mathcal A.
$$
Repeating with $f$ replaced by $f/m,$ where $m$ is a positive integer, we get
$$
	\{\omega: f(\omega,\cdot)\in mC\}\in \mathcal A \quad (m\ge 1).
$$
Finally, we deduce that
$$
	\{\omega:f(\omega,\cdot)\in A^2(U)|_K\} =   
		\bigcup_{m\ge 1}\{\omega: f(\omega,\cdot)\in mC\}       \in\mathcal A.
$$
\end{proof}

The next lemma is a sort of measurable identity principle.

\begin{lemma}\label{domain}
Let $D$ be a domain in $\C,$ let  $(\Omega,\mathcal A)$ be a measurable space, and let $f:\Omega\times D\to\C$ be
a function such that $f(\omega,\cdot)$ is holomorphic in $D$ for each $\omega\in\Omega.$ Define
$$
	M:= \{z\in D: f(\cdot,z)\text{~is measurable}\}.
$$
If $M$ has a limit point in $D,$ then $M = D.$
\end{lemma}

\begin{proof}
In what follows, we write $f^{(n)}(\omega,z)$ for the $n$-th derivative of $f(\omega,z)$ with respect to $z$. 
Define
\[
N:=\{z\in D: f^{(n)}(\cdot,z) \text{~is measurable for all~}n\ge0\}.
\]
We first show that each limit point of $M$ in $D$ belongs to $N$. 
Let $z^*$ be such a limit point, say $z^*=\lim_{k\to\infty}z_k$, where $(z_k)$ are points in $M\setminus\{z^*\}$.
By Taylor's theorem, for each $n\ge0$ and each $\omega\in\Omega$, we have
\[
f^{(n)}(\omega,z^*)/n! = 
\lim_{k\to\infty}\frac{f(\omega,z_k)-\sum_{m=0}^{n-1}f^{(m)}(\omega,z^*)/m!(z_k-z^*)^m}{(z_k-z^*)^n}.
\]
By induction on $n$, it follows that  $\omega\mapsto f^{(n)}(\omega,z^*)$ is measurable for all $n\ge0$.
In other words, $z^*\in N$.

Next, we show $N$ is open in $D$. Let $z_0\in D$ and 
$r:=\text{dist}(z_0,\partial D)$. Then, for all $n\ge0$ and $|z-z_0|<r$ and $\omega\in\Omega$, we have
\[
f^{(n)}(\omega,z)=\sum_{m\ge 0} \frac{f^{(n+m)}(\omega,z_0)}{m!}(z-z_0)^m.
\]
Hence, if $z_0\in N$, then $z\in N$ for all $z$ with $|z-z_0|<r$.
So $N$ is indeed open in $D$.

Lastly we show that $N$ is closed in $D$. If $z_k\to z_0$ in $D$, 
then, for each $n\ge0$ and $\omega\in\Omega$, we have 
\[
f^{(n)}(\omega,z_0)=\lim_{k\to\infty}f^{(n)}(\omega, z_k).
\]
Hence, if $z_k\in N$ for all $k$, then also $z_0\in N$.
Thus $N$ is indeed closed in $D$.

To conclude: if $M$ contains a limit point in $D$, then $N$ is non-empty, open and closed in $D$, so,
as $D$ is connected, we must have $N=D$. Clearly $N\subset M$, hence also $M=D$.
\end{proof}

\begin{lemma}\label{U}
Let $K$ be a compact subset of $\C$, let $U$ be an open neighborhood of $K$, let $(\Omega,\mathcal A)$ be  a measurable space and let $f : \Omega\times \C\to\C$ be a function such that: 
\begin{itemize}
\item $f(\cdot,z)$ is measurable for all $z\in K,$
\item  for each $\omega\in\Omega$, there exists a function $g_\omega$ holomorphic in $U$ such that $g_\omega|K=f(\omega,\cdot)$.
\end{itemize}
Then there exists a function $F : \Omega\times U\rightarrow \C$ such that:
\begin{itemize}
\item $F(\cdot,z)$ is measurable for all $z\in U,$
\item $F(\omega,\cdot)$ is holomorphic in $U$ for all $\omega\in\Omega$,
\item $F|_{\Omega\times K} = f.$
\end{itemize}
\end{lemma}

\begin{proof}
It is enough to construct $F$ on $\Omega\times D$ for each connected component $D$ of $U.$ There are three cases to consider. 

If $K\cap D =\emptyset,$ then we can simply take  $F\equiv 0.$ 

If $K\cap D$  is non-empty and finite, say $K\cap D  = \{z_1,\ldots, z_n\},$ then we fix polynomials $p_1,\ldots,p_n$
such that $p_j(z_k) = \delta_{jk},$ and define
$$
	F(\omega,z) := \sum_{j=1}^n f(\omega, z_j)p_j(z)
		\quad ((\omega,z)\in\Omega\times U).
$$
It is easy to see that this $F$ has the required properties.

Finally, if $K\cap D$ is infinite, then, for each $\omega\in\Omega,$ we define $F(\omega,\cdot)$ to be the holomorphic extension of $f(\omega,\cdot)$ to $D,$ which exists by hypothesis, and is unique by the identity  principle. By Lemma~\ref{domain}, the
function $F(\cdot,z)$ is measurable for all $z\in D.$ The other two properties required of $F$ are clear.
\end{proof}

We now have all the ingredients needed to prove Theorem~\ref{T:extension}.

\begin{proof}[Proof of Theorem~\ref{T:extension}]
For each $n\ge 1,$ define
\begin{align*}
U_n &:= \{z\in \C : \text{dist}(z,K) < 1/n\},\\
\Omega_n &:= \{\omega\in\Omega: f(\omega,\cdot) \in A^2(U_n)|_K\}.
\end{align*}
By Lemma \ref{Bergman}, we have $\Omega_n\in \mathcal A.$ Thus, if we set $\mathcal A_n := \{A\cap\Omega_n: A\subset\mathcal A\},$ then $(\Omega_n,\mathcal A_n)$ is a
measurable space. By Lemma~\ref{U}, there exists a function $F_n : \Omega_n\times U_n \rightarrow\C$ such that
\begin{itemize}
\item for each $z\in U_n,$ the function $F_n(\cdot, z)$ is measurable,
\item for each $\omega\in \Omega_n,$ the function $F_n(\omega,\cdot)$ is holomorphic on $U_n$,
\item $F_n = f$ on $\Omega_n\times K.$
\end{itemize}
Clearly we have 
$\Omega_1\subset\Omega_2\subset\Omega_3\subset\cdots,$ and the hypothesis (ii)$'$ of the theorem implies that $\cup_{n\ge 1}\Omega_n=\Omega.$ Define $F : \Omega\times\C\rightarrow\C$ by
$$
	F:= \sum_{n\ge 1}F_n1_{(\Omega_n\setminus\Omega_{n-1})\times U_n},
$$
where, for convenience, we have set $\Omega_0:=\emptyset.$ It is routine to check that this function $F$ satisfies the conclusions of the theorem.
\end{proof}

\begin{remark} The Bergman spaces $A^2(U)$ play only an auxiliary role
in this proof. Any other Banach space $H(U)$ of holomorphic functions on $U$ would do just as well, provided that $H(U)$ is reflexive, that $H(U)$ contains all functions that are holomorphic on a neighborhood of $\overline{U}$, and that convergence in $H(U)$ implies uniform convergence on compact sets.
\end{remark}


\section{Random compact sets} 

Because of the Oka--Weil Theorem, the most important notions in complex approximation in several variables are those of polynomial  or rational convexity of compacta.  
First we state a few basic properties of random compact sets  and then show that the image of a  compact set by a random continuous function  is a random compact set. We also show that the polynomially and rationally convex hulls are transformations that preserve randomness, and that the Siciak extremal function and pluricomplex Green function of a random compact set are random functions.

For a metric space $(X,d)$, we denote by 
$\mathcal K(X)$ the family of compact subsets of $X$ and by 
$(\mathcal K^\prime(X),d_H)$ the space of {\em non-empty} compact subsets of $X$ equipped with the Hausdorff distance $d_H.$  We recall the following useful property. 

\begin{lemma}\label{ksep}
The spaces $\mathcal K^\prime(\R^n)$ and $\mathcal K^\prime(\C^n)$ are separable.
\end{lemma}

Let $g: X \rightarrow Y$ be a continuous function. We may extend it to a function $g^{\mathcal K^\prime}:\mathcal K^\prime(X)\rightarrow \mathcal K^\prime(Y),$  defined by setting $g^{\mathcal K^\prime}(Q):=g(Q),$ for each $Q\in\mathcal K^\prime(X).$  The following is well known.

\begin{lemma}\label{continuous extension}
Let $X$ and $Y$ be metric spaces. If $g:X\rightarrow Y$ is a continuous function on $X,$ then the extension $g^{\mathcal K^\prime}:\mathcal K^\prime(X)\rightarrow \mathcal K^\prime(Y)$ is a continuous function. 
\end{lemma}

\begin{remark}\label{characterizations}
Since $\mathcal K^\prime(X)$ is a metric space, we may also consider it as a 
measurable 
space, where the measurable sets are the Borel subsets of $\mathcal K^\prime(X).$ 
We may also, as shown in Theorem D.6 of \cite{Mo}, characterize $\mathcal B(\mathcal K^\prime(X))$ as the $\sigma$-algebra generated by the sets $\{K\in \mathcal K^\prime(X) : K\subset G\}$ where  $G$ varies over the open sets of $X$. Alternatively, the Borel sets can be generated by the sets $\{K\in \mathcal K^\prime(X) : K\cap G\neq \emptyset\}$ where again $G$ varies over the open sets of $X$.
\end{remark}

A random compact set  in $X$ is a  measurable function $k:\Omega\rightarrow\mathcal K^\prime(X).$ 
If $k_j:\Omega\rightarrow \mathcal K^\prime(X), \, j=1,2,$ are two random compact sets, we denote by  $k_1\cup k_2$ the function $\Omega\rightarrow\mathcal K^\prime(X),$ defined by $(k_1\cup k_2)(\omega):=k_1(\omega)\cup k_2(\omega),$ for $\omega\in\Omega.$ 

Using the characterization of Remark \ref{characterizations}, one can prove the following lemmas.

\begin{lemma}\label{cup}
If $k_j:\Omega\rightarrow \mathcal K^\prime(X), \, j=1,2,$ are two random compact sets, then $k_1\cup k_2$ is a random compact set. We also have the countable intersection of random compact sets  is a random compact set. 
\end{lemma}

\begin{lemma}\label{countable cup}
Let $\{k_i\}_{i=0}^\infty$ be random compact sets. Suppose we know that for each $\omega$, $k(\omega):=\cup_{i=0}^\infty k_i(\omega)$ is a compact set. Then $k$ is a random compact set.
\end{lemma}

By Corollary \ref{maps F to f}, if $X$ is a compact metric space, we can say that a function $f: \Omega \times X \rightarrow \C^n$ is a random element of $C(X, \C^n),$  if $f(\omega, \cdot) \in C(X, \C^n)$ for all $\omega \in \Omega$ and $f(\cdot, z)$ is measurable for all $z \in X$.

\begin{lemma}\label{singleton}
Let $X$ be a compact metric space and  $f:\Omega\times X\rightarrow \C^n$ be  a random element of $C(X,\C^n).$ 
Then, for each $x\in X,$ the function 
$$
	f(\cdot,x):\Omega\rightarrow\C^n, \quad \omega\mapsto f(\omega,x)
$$ 
is a random complex vector and the (singleton-valued) function
$$
	\mathcal K^f(\cdot,\{x\}):\Omega\rightarrow\mathcal K^\prime(\C^n), \quad 
		\omega\mapsto \{f(\omega,x)\}
$$
is a random compact set. 
\end{lemma}

\begin{proof}
The first assertion follows from the definition of a random  function. 

For brevity, we write $F_x=\mathcal K^f(\cdot,\{x\}).$ We need to show that $F_x$ is measurable. We use separability arguments. For $W\in \mathcal K^\prime(\C^n)$ and  $r>0,$ consider the closed ball $\overline{B_r}(W)=\{V\in \mathcal K^\prime(\C^n) : d_H(W,V)\leq r\}$  in $\mathcal K^\prime(\C^n)$. 
It is easy to see that every closed ball is closed.

Denote by $\widetilde x:C(X,\C^n)\rightarrow  \mathcal K^\prime(\C^n)$ 
the mapping $g\mapsto \{g(x)\},$ for $g\in C(X,\C^n).$ We then see directly that
\begin{align*}
    \widetilde x^{-1}(\overline{B_r}(W))&=\{g \in C(X,\C^n) : d_H(g(\{x\}),W)\leq r \}\\
    &=\bigcap_{y\in W}\{g \in C(K,\C^n) : d(g(x),y)\leq r \}.
\end{align*}
As $\C^n$ is separable and $W\subset \C^n$, $W$ is also separable. There thus exists a countable dense subset $W^*$ of $W$ and, since $d$ is a continuous function, we can restrict the intersection to $W^*$. We also know that the $\{g\in C(X,\C^n) : d(g(x),y)\leq r\}$ are measurable sets, as they are closed. As such, $ \widetilde x^{-1}(\overline{B_r}(W))$ is measurable, being the countable intersection of measurable sets. We know from Lemma \ref{ksep} that $\mathcal K(\C^n)$ is separable. Thus, by a similar argument to that in the proof of Proposition \ref{measurable element}, every open set of $\mathcal K^\prime(\C^n)$ can be expressed as a countable union of closed balls, and we can then generalize to all Borel subsets. Hence,  the function $\widetilde x$ is a random function.  
By hypothesis, $f$ is a random element of $C(X,\C^n),$ so the mapping $\omega\mapsto f(\omega,\cdot)$ is measurable. 
It follows that $F_x$ is measurable, since it is  the composition 
$\widetilde x(f(\omega,\cdot))$ of measurable functions: 
$$
	  F_x(\omega) = 
		\mathcal K^f(\omega,\{x\} ) =  \{f(\omega,x)\} =  \widetilde x(f(\omega,\cdot)).
$$

\end{proof}

Let  $X$ be a compact metric space and  $g:\Omega\times X\rightarrow\C^n$  a random continuous function  on $X.$ For $\omega\in\Omega,$ we denote $X^g(\omega)=g^{\mathcal K^\prime}(\omega,X).$

\begin{theorem}\label{F sub f}
Suppose $X$ is a compact metric space. If $g:\Omega\times X\rightarrow\C^n$ is a random continuous function  on $X$, then $X^g(\omega)$ is a random compact set in $\C^n.$ 
\end{theorem}

\begin{proof}
Let $g:\Omega\times X\rightarrow \C^n$ be a random continuous function. 
We need only show that $X^g$ is  measurable. Let $\{x_1,x_2,\ldots\},$ be a countable (ordered) dense subset of $X;$  let $X_j$ be the set consisting of the first $j$ elements of this set; and let $g_j(\omega,x)$ be the restriction of $g(\omega,x)$ to $\Omega\times X_j.$
Clearly, for each $j,$ the function $g_j$ is a random continuous function on $X_j.$  For each $j=1,2,\cdots,$ we define  the set-function
$$
	F_j:\Omega\rightarrow \mathcal K^\prime(\C^n), \quad  \omega\mapsto g_j(\omega, X_j). 
$$
We shall show that $F_j$ is a random compact set.

By Lemma \ref{singleton}, the function $F_1$ is itself measurable (and analogously for all 1-element sets).  $F_1:\Omega\rightarrow\mathcal K^\prime(\C^n)$ is  random-continuous  compact set. 

Denote by $C(X,\C^n)$ the set of continuous functions from $X$ to $\C^n$ and,
for $x\in X,$  denote by $\widetilde x:C(X,\C^n)\rightarrow  \mathcal K^\prime(\C^n)$ 
the mapping $h\mapsto \{h(x)\},$ for $h\in C(X,\C^n).$
Again, by Lemma \ref{singleton}, each $\widetilde x_j$ is measurable and then, applying Lemma \ref{cup} $j$~times, we obtain that $F_j$ is a random compact set. 
Since $g(\omega,\cdot)$ is continuous for each $\omega$, its extension $g^{\mathcal K^\prime}(\omega,\cdot)$ to compact sets is continuous. Since $X_j\rightarrow X,$ one can then enter the limit into the function : 
$$
	\lim_{j\rightarrow \infty}F_j(\omega)=
	\lim_{j\rightarrow \infty}g^{\mathcal K^\prime}(\omega, X_j)=
	g^{\mathcal K^\prime}(\omega,\lim_{j\rightarrow \infty}X_j) = 
	g^{\mathcal K^\prime}(\omega,X)=X^g(\omega).
$$ 
We have shown that the sequence  $F_j$ of  measurable functions converges  pointwise to the function $X^g.$  Since these functions take their values in a metric space, it follows easily   
that $X^g$ is  measurable, which concludes the proof. 
\end{proof}

We define a compact transformation as a function $T : \mathcal K'(X)\rightarrow \mathcal K'(X)$. We say that a compact transformation is randomness-preserving if 
$K$ being a random compact set implies that $T(K)$ is a random compact set. It is important that this property not depend on a specific structure of $\Omega$, it must work for every measurable space.

\begin{lemma}\label{measurableness-preserving}
    A compact transformation is randomness-preserving if and only if it is a random function.
\end{lemma}
\begin{proof}
    Suppose a compact transformation $T$ is randomness-preserving. This means that for any choice of measurable events set $(\Omega,\mathcal A)$, the measurability of the mapping $\omega \mapsto K(\omega)$ implies the measurability of the mapping $\omega\mapsto T(K(\omega))$.
    
We may thus choose $\Omega=\mathcal K'(X)$ and $\mathcal A=\mathcal B(\mathcal K'(X))$. We now look at the 
identity mapping $I:\Omega\rightarrow \mathcal K'(X)$  mapping $K$ to $K.$ This mapping defines a random compact set, since the identity function is certainly measurable. 
Thus, as $T$ is randomness-preserving, the mapping $K\mapsto T(I(K))=T(K)$ is random.

    The converse follows directly by composition of measurable functions.
\end{proof}

For $K\in \mathcal K^\prime(\C^n),$ we denote by $\widehat K$ the polynomially convex hull of $K.$ It is defined as $$\widehat K = \{z \in \C^n : \abs{p(z)} \leq \max_{x\in K} \abs{p(x)} \text{ } \forall p \in P(\C^n)\}$$where $P(\C^n)$ is the set of complex polynomials from $\C^n$ to $\C$. We shall show that the polynomially convex hull of a random compact set is also a random compact set, but first we present some lemmas.

\begin{lemma} \label{lemmaPolyHull}
Let $K$ be a compact subset of $\C^n$ and let $P^\Q(\C^n)$ denote the polynomials in $\C^n$ whose coefficients have rational real and imaginary part. Then,
\begin{equation} \label{hullQ}
    \widehat K = \{z \in \C^n : \abs{p(z)} \leq \max_{x\in K} \abs{p(x)} \text{ } \forall p \in P^\Q(\C^n)\}.
\end{equation}
\end{lemma}

Let $W$ be a measurable space.
We define as a pseudo-random compact set  in $X$ a mapping 
$$f : W \rightarrow\mathcal K(X),$$ 
if pre-images of Borel subsets of $\mathcal K^\prime(X)$ are measurable. 
Here $W$ can be a metric space with the Borel sets, or $W=\Omega$ . 
We see that every random compact set in $\C^n$  is a pseudo-random compact set. If $f$ is a pseudo-random compact set and $f^{-1}(\{\emptyset\})$ is empty, then $f$ is a random compact set (here we see $f^{-1}(\{\emptyset\})$ as the pre-image of the empty compact set, and not as the the pre-image of an empty set of compact sets). 
We have directly that the composition of a measurable function followed by
a pseudo-random compact set  is a pseudo-random compact set.

We define as a pseudo-continuous compact-valued function a function $f: W\rightarrow\mathcal K(\C^n)$, where $W$ is a metric space, if, for all $w\in W$ such that $f(w)\neq \emptyset$ and $\varepsilon>0$, there exists $\delta>0$ such that, if $d(w,y)<\delta$ and $f(y)\neq \emptyset$, then $d_H(f(w),f(y))<\varepsilon$, and if moreover the set $f^{-1}(\emptyset)$ is a measurable set.

\begin{lemma}\label{pseudo union}
	If $\{K_i\}_{i=0}^\infty$ is a sequence of pseudo-random compact sets, then $k(\omega):=\bigcap_{i=0}^\infty K_i(\omega)$ is a pseudo-random compact set. If we know that $K(\omega):=\bigcup_{i=0}^\infty K_i(\omega)$ is a compact set for each $\omega$, then $K$ is a pseudo-random compact set.
\end{lemma}

\begin{proof}
	The proof follows from the fact that we supposed that the pre-images of Borel subsets are measurable. We may use a similar proof method to the one employed in Lemmas \ref{cup} and \ref{countable cup}.

\end{proof}

\begin{lemma}\label{pseudo cont}
Let $f : X\rightarrow\mathcal K(\C^n)$ be a pseudo-continuous compact-valued function. Then it is also a pseudo-random compact set.
\end{lemma}

\begin{proof}
    By our definition of pseudo-continuous function, we have that the restriction $f|_{X\setminus f^{-1}(\emptyset)}$ is a continuous function over $X\setminus f^{-1}(\emptyset).$
Thus, if $O$ is an open set of $\mathcal K (\C^n),$ then $f^{-1}(O)$ is an open set relative to $X\setminus f^{-1}(\emptyset)$. 
From the definition of pseudo-continuity, 
$f^{-1}(\emptyset)$ is measurable in $X,$ so $X\setminus f^{-1}(\emptyset)$ is also measurable. We have that  $f^{-1}(O)$ is a measurable subset of the measurable set  $X\setminus f^{-1}(\emptyset).$ 
Thus, $f^{-1}(O)$ is measurable relative to $X.$
As the open sets generate the Borel subsets, all pre-images of measurable subsets of $\mathcal K(\C^n)$ are measurable.
\end{proof}

\begin{theorem}\label{P-hull measurable}
Let $K$ be a random compact subset of $\C^n.$  Then its polynomially convex hull $\widehat K$, defined pointwise as $\widehat K(\omega):=\widehat{K(\omega)}$, is a random compact set.
\end{theorem}

\begin{proof} 
   Let $N_p^K := \{z : \abs{p(z)} \leq \max_{x \in K} \abs{p(x)}\}$. Then, by  Lemma \ref{lemmaPolyHull},
    \begin{equation*}
        \widehat K(\omega) = \bigcap_{p \in P^\Q(\C^n)} N_p^{K(\omega)}.
    \end{equation*}
    The function $|p(\cdot)|$ is  continuous, and thus
$$N_p^{K(\omega)}=|p(\cdot)|^{-1}(\overline B_{\max_{x \in K(\omega)} \abs{p(x)}}(0))$$
is a closed set, as the pre-image of a closed set. But $N_p^K$ is not necessarily bounded. To solve this, we may write :
      \begin{equation}\label{khat}
        \widehat K(\omega) = \bigcup_{i=1}^\infty \bigcap_{p \in P^\Q(\C^n)} 						\left(N_p^{K(\omega)}\cap \overline{B}_i(0)\right).
    \end{equation}
This is true since for every $\omega$, if $i$ is  great enough, the compact set $\widehat K(\omega)$ is contained in the ball of radius $i$. Each $N_p^{K(\omega)}\cap \overline{B}_i(0)$ is compact, since it is  the intersection of a closed set with a compact set.
    
    Choose a non-constant polynomial $p$ with rational coefficients and a natural number $i$. We shall show that the mapping $\omega\mapsto N_p^{K(\omega)}\cap \overline{B}_i(0)$ is a pseudo-random compact set. It will follow from Lemma \ref{pseudo union} that $\bigcap_{p \in P^\Q(\C^n)} (N_p^{K(\omega)}\cap \overline{B}_i(0))$ is a pseudo-random compact set, since the intersection is countable. Then, as we already know that $\widehat K(\omega)$ is compact and non-empty, we can apply Lemma \ref{pseudo union} to prove the theorem.
    
   We consider several auxiliary  mappings. We already know the mapping $\omega\mapsto K(\omega)$ is measurable, by hypothesis.
As $|p(\cdot)|$ is continuous, we know  that $K\mapsto |p|(K)$ is a continuous mapping from $\mathcal K^\prime (\C^n)$ to $\mathcal K^\prime(\R)$. The mapping $K\mapsto \max\{x:x\in K\}$ between $\mathcal K^\prime(\R)$ and $\R$ is also continuous. Let $\varepsilon>0$ and fix $K_0\in \mathcal K^\prime(\C^n)$. Suppose $d_H(K_0,K) < \varepsilon$. Then, for each $x \in K_0$ (respectively $x\in K$) there is a $y \in K$ (respectively $y\in K_0$) such that  $\abs{x-y} < \varepsilon$. It follows that
    \begin{equation*}
        |\max\{x:x\in K_0\}-\max\{x:x\in K\}|<\varepsilon.
    \end{equation*}

Denote $W_p^M:=\{z : |p(z)|\leq M\}$. We shall now show that the mapping 
$W: \R \rightarrow \mathcal{K}(\C^n)$ that maps $M$ to $W_p^M \cap \overline B_i(0)$ is pseudo-random. Denote, for $O$ open in $\C^n$, the sets 
\begin{equation*}
H_{O} := \{K \in \mathcal{K}^\prime(\C^n) : K \cap O\neq \emptyset \},
\end{equation*}
which generate the Borel subsets of $\mathcal K^\prime(\C^n)$. 
Since $W(M) \subset W(N)$ whenever $M \leq N$ (by continuity of $\abs{p}$), we have that
\begin{equation*}
\{M : W(M) \in H_{O}\} = \{M : W(M) \cap O\neq \emptyset\} \in \{[M_0, \infty), (M_0, \infty),\emptyset\},
\end{equation*}
for some $M_0 \in \R$. The mapping $W$ is pseudo-random as a mapping from $\R$ to $\mathcal K(\C^n).$

We can now see that the mapping $\omega \mapsto N_p^{K(\omega)}\cap \overline B_i(0)$ is in fact 
the result of composing the sequence of mappings
    \begin{equation*}
        \omega\mapsto K(\omega)\mapsto |p(K(\omega))|\mapsto \max\{x \in |p(K(\omega))|\}\mapsto W_p^{\max\{x \in |p(K(\omega))|\}}\cap \overline B_i(0)
    \end{equation*}
    and is therefore pseudo-random. Being always non-empty, it is measurable.
    
As noted earlier, this proves the theorem.
\end{proof}

\begin{corollary}\label{map P-hull measurable} 
The mapping 
$$
	\mathcal K^\prime(\C^n)\rightarrow\mathcal K^\prime(\C^n), \quad K\mapsto \widehat K
$$
is measurable.
\end{corollary} 

\begin{proof} This follows from Lemma \ref{measurableness-preserving} and Theorem \ref{P-hull measurable}.
\end{proof}

For a random compact set  $K:\Omega\rightarrow\mathcal K^\prime(\C^n),$ we define the random polynomially convex hull as the function $M\widehat K:\Omega\rightarrow \mathcal K^\prime(\C^n),$ given by setting 
$$
	M\widehat K(\omega)=	\{z: |p(\omega,z)|\le 
		\max_{x\in K(\omega)}|p(\omega,x)|, \forall \, \mbox{random polynomials} \,\, p\}.		
$$
Because of the  central role of the polynomially convex hull in the Oka--Weil Theorem, the random polynomially convex hull should be investigated in relation to approximation by random polynomials. The following proposition is rather obvious.

\begin{proposition}
    Let $K$ be a random compact set. Then, $\widehat K=M\widehat K$.
\end{proposition}

Let us denote by $ \mathcal R\hull\,K$ the rationally convex hull of a compact set $K.$ It is defined as$$
\mathcal R\hull\,K := \{z \in \C^n : \abs{r(z)} \leq \max_{x \in K}  \abs{r(x)} \text{ } \forall r \in \mathcal R_K(\C^n)\}$$
where $\mathcal R_K(\C^n)$ is the set of rational functions from $\C^n$ to $\C$ which are holomorphic over $K$.

\begin{lemma}
Let $K$ be a compact subset of $\C^n$ and let $\mathcal \mathcal R_K^\Q(\C^n)$ denote the rational functions in $\C^n$ without singularities on $K$ and whose coefficients have rational real and imaginary parts. Then,
\begin{equation} \label{rationalHullQ}
    \mathcal R\hull\,K = \{z \in \C^n : \abs{r(z)} \leq \max_{x \in K}  \abs{r(x)} \text{ } \forall r \in \mathcal R_K^\Q(\C^n)\}.
\end{equation}
\end{lemma}

The following lemma is a known fact that we simply recall.

\begin{lemma} \label{pole set}
	Let $r : \C^n\rightarrow \C$ be a  rational function. Then the set $S(r)$ of singularities of $r$ is a closed set.
\end{lemma}

The proof of the next lemma is very simple and is thus left to the reader to verify.

\begin{lemma} \label{restriction measurable}
	Let $X$ be a metric space. Suppose  $X=A\cup B,$ where $A$ and $B$ are disjoint measurable subsets. Let $Y$ be a metric space and $f : X\rightarrow Y$ a function such that its restriction $f_A$ to the metric space $A$ is measurable, and its restriction $f_B$ to the metric space $B$ is measurable. Then $f$ is a measurable function.
\end{lemma}

\begin{theorem}\label{R-hull measurable}
Let $K$ be a random compact set  in $\C^n$. Then, its rationally convex hull $\mathcal R\hull\,K$, defined pointwise as 
$$
[\mathcal R\hull\,K](\omega) :=  \mathcal R\hull\,(K(\omega)),
$$
 is a random compact set.
\end{theorem}

\begin{proof}
Let $r$ be a rational function and $K$ be a non-empty compact subset of $\C^n$. 
Set $A_r=\{K\in \mathcal K^\prime(\C^n) : K\cap S(r)\neq \emptyset\}$ and $B_r=\mathcal K^\prime(\C^n)\setminus A_r.$  Then, $\mathcal K^\prime(\C^n)$ is the disjoint union of $A_r$ and $B_r.$ 
We define a function $\mu_r : \mathcal K'(\C^n) \rightarrow \mathcal K'(\C^n)$ by
\begin{equation*}
    \mu_r(K) :=
    \begin{cases}
    \widehat K, & \text{if } K\in A_r, \\
    \widehat K \cap N_r^K, & \text{if } K\in B_r,
    \end{cases}
\end{equation*}
where $N_r^K = \{z \in \C^n : \abs{r(z)} \leq \max_{x \in K} \abs{r(x)}\}$. 
We note that, for $K\in B_r,$ the set $N_r^K$ is never empty, since it contains $K.$ 
Since all polynomials are also rational functions, $\mathcal R\hull\,K\subseteq \widehat K.$
Therefore, we can write
\begin{equation*}
     \mathcal R\hull\,K =  (\mathcal R\hull\,K)\cap\widehat K = 
	\bigcap_{r \in \mathcal R^\Q(\C^n)} \mu_r(K),
\end{equation*}
where $\mathcal R^\Q(\C^n)$ is the set of all rational functions in $\C^n$ whose coefficients have rational real and imaginary parts. The intersection is countable since $\mathcal R^\Q(\C^n)$ can be viewed as a subset of $P^\Q(\C^n) \times P^\Q(\C^n)$.  

It is thus sufficient to show that $K\mapsto \mu_r(K)$ is measurable, for then by Lemma~\ref{cup}, 
the mapping $\omega \mapsto  \mathcal R\hull\,K(\omega)$ will be measurable.

First of all, let us show that the subset $A_r\subset \mathcal K^\prime (\C^n)$  is measurable. 
We may assume that $A_r\not=\emptyset$ and hence $S(r)\not=\emptyset,$  
We construct a sequence of open subsets of $\C^n$ defined as 
$$O_n=\left\{z : d(z,S(r))<\frac{1}{n}\right\}.$$
Then, $O_{n+1}\subseteq O_n$ and  $\bigcap_{n=1}^\infty O_n = S(r)$. We have that$$A_r=\bigcap_{n=1}^\infty \{K : K\cap O_n\neq \emptyset\}.$$ By our second characterization of measurable sets (see Remark \ref{characterizations}), this is a countable intersection of measurable sets, and thus measurable. We have shown that $A_r$ and hence also $B_r=\mathcal K^\prime(\C_n)\setminus A_r$ are measurable.

The definition of the  function $\mu_r(K)$ depends on whether $K$ is in $A_r$ or $B_r.$ 
Since $\mathcal K^\prime(\C^n) = A_r\cup B_r$ is the union of disjoint measurable sets, in order to assure that $\mu_r$ is measurable, it suffices, by Lemma \ref{restriction measurable}, to show that the restrictions of $\mu_r$ to  $A_r$ and  to $B_r$ are measurable . 

By Corollary \ref{map P-hull measurable},
 we have that the mapping $K\mapsto \widehat K$ is measurable, and thus its restriction to $A_r,$ which is the same as the restriction of $\mu_r$ to $A_r,$ is also measurable. 

We claim that the mapping $K\mapsto \widehat K\cap N_r^K$, is measurable over $B_r$. We first notice that over the elements of $B_r$, our rational function $r$ is in fact continuous.
With essentially the same proof as for the polynomially convex hull (which did not use any special properties of polynomials other than that they are continuous functions), one can show that, over $B_r$, the mappings $K\mapsto N_r^K \cap \overline B_i(0)$ are pseudo-random compact. It then follows from Lemma \ref{pseudo union}, as the mapping $K\mapsto \widehat K$ is measurable,  that each mapping $K\mapsto N_r^K \cap \overline B_i(0)\cap \widehat K$ is pseudo-random compact. We then have that the mapping $$
	K\mapsto N_r^K \cap  \widehat K=\bigcup_{i=1}^\infty N_r^K \cap \overline B_i(0)\cap \widehat K
 $$
is pseudo-random compact by Lemma \ref{pseudo union}
and in fact measurable, as $N_r^K \cap  \widehat K$ 
is always non-empty.

By Lemma \ref{restriction measurable}, the mapping $\mu_r$ is measurable. 
\end{proof}

	One can use random compact sets  to construct interesting random functions.	
An important function for complex approximation (see \cite{Ni}) is the pluricomplex Green function $V_K$ for a non-empty compact set $K\subset\C^n$. It is defined as $V_K(z):=\log \Phi_K(z),$ for $z\in\C^n,$ where
$\Phi$ is the  Siciak extremal function, defined as
	\begin{equation*}
		\Phi_K(z):=\sup\Bigl\{  |p(z)|^{1/\deg p}: p\in P(\C^n),~\|p\|_K\leq 1,~\deg(p)\geq 1\Bigr\}.
	\end{equation*}

We have the following result (see also \cite{BL}).	
	\begin{theorem}\label{siciak}
		The Siciak extremal function and the pluricomplex Green function of a random compact set  $K$ are random functions. That is, the functions $(\omega,z)\mapsto\Phi_{K(\omega)}(z)$ and  $V_{K(\omega)}(z)=\log \Phi_{K(\omega)}(z),$ are random functions into $\R\cup\{\infty\}$.
	\end{theorem}

	\begin{proof}
It is sufficient to show that the Siciak extremal function is a random function. 
		Let $p$ be a polynomial. We construct the function
		\begin{align*}
			g_p(\omega,z):=
			\begin{cases}
       			|p(z)|^{1/\deg p}, &\quad \|p\|_{K(\omega)}\leq 1,\\
       			0, &\quad \|p\|_{K(\omega)}>1,\\
     		\end{cases}
		\end{align*}
		which evaluates our polynomial if it respects the condition $\|p\|_{K(\omega)}\leq 1$, and returns $0$ if not.
		
		We first prove that this function is a random function. Fix $z\in \C^n$. Let $M$ be a measurable subset of $\R$. We then have
		\begin{align*}
			\{\omega : g_p(\omega,z)\in M\} & =(\{\omega : \|p\|_{K(\omega)}>1\}\cap\{\omega : 0\in M\})\\
&\quad\cup (\{\omega : \|p\|_{K(\omega)}\leq1\}\cap\{\omega : |p(z)|^{1/\deg p}\in M\}).
		\end{align*}
		We remark that the sets $\{\omega : 0\in M\}$ and $\{\omega : |p(z)|^{1/\deg p}\in M\}$ are always $\emptyset$ or $\Omega$, as the conditions do not depend on $\omega$. The set $\{\omega:\|p\|_{K(\omega)}\le 1\}$ is measurable, because the function $|p(K(\omega))|$ is measurable, since it is the composition of the measurable function $K:\Omega\rightarrow\mathcal K^\prime(\C^n)$ with the continuous function $\mathcal K^{|p|}:\mathcal K^\prime(\C^n)\rightarrow \mathcal K^\prime(\R)$. Thus, the set $\{\omega : g_p(\omega,z)\in M\}$ is measurable. This proves that $g_p(\cdot,z)$ is measurable, which means that $g_p$ is a  random function.
		
		Let us now return to the Siciak extremal function. We shall prove that it may be written as
		\begin{equation*}
			\Phi_K(z)=\sup\Bigl\{  |p(z)|^{1/\deg p} : p\in P^\Q(\C^n),~ \|p\|_K\leq 1, ~\deg(p)\geq 1 \Bigr\},
		\end{equation*}
		where $P^\Q(\C^n)$ is the set of all polynomials with rational coefficients.
		
		In order to do this, it is sufficient to show that for each $\varepsilon>0$, each $p\in P(\C^n)$ and each $z\in \C^n$, there is a polynomial $q\in P^\Q(\C^n)$ such that 
		$$
		\Bigl||p(z)|^{1/\deg p}-|q(z)|^{1/\deg q}\Bigr|<\varepsilon.
		$$
Suppose that $p$ has degree $k$. We will choose $q$ of degree $k$ also. Using the fact that if $x>y\geq 0$, then $x^{1/k}-y^{1/k}\le(x-y)^{1/k}$, and the reverse triangle identity $||x|-|y||\leq|x-y|$, we have that
		\begin{align*}
			||p(z)|^{1/\deg p}-|q(z)|^{1/\deg q}| &= ||p(z)|^{1/k}-|q(z)|^{1/k}|\\
			&\leq ||p(z)|-|q(z)||^{1/k}\\
			&\leq |p(z)-q(z)|^{1/k}.
		\end{align*}
		This reduces the problem to finding $q$ such that $|p(z)-q(z)|<\varepsilon^k$. But we know this is possible, as we can approximate $p$ by choosing $q$ with very close rational coefficients.
		
		But now, by the definition of $g_p$, we have that
		\begin{align*}
			\Phi_{K(\omega)}(z)=\sup_{p\in P^\Q(\C^n) : \deg(p)\geq 1} g_p(\omega,z).
		\end{align*}
		We were able to remove the condition over our polynomials that $\|p\|_{K(\omega)}\leq 1$ since $g_p(\omega,z)$ always returns $0$ if the condition is not fulfilled.
		
		We now prove that this function is measurable, that is,  that $\Phi_{K(\cdot)}(z)$ is measurable. Consider the subset $(-\infty,a]$  of $\R$. We have that
		\begin{align*}
			\{\omega : \Phi_{K(\omega)}(z)\in (-\infty,a]\}
			&=\{\omega : \Phi_{K(\omega)}(z)\leq a\}\\
			&=\Bigl\{\omega : \sup_{p\in P^\Q(\C^n) : \deg(p)\geq 1} g_p(\omega,z)\leq a\Bigr\}\\
			&=\{\omega :g_p(\omega,z)\leq a, \forall p\in P^\Q(\C^n), \deg(p)\geq 1\} \\
			&=\bigcap_{p\in P^\Q(\C^n) : \deg(p)\geq 1}\{\omega : g_p(\omega,z)\leq a\}.
		\end{align*}
		As $g_p$ is a random function, we know that each of the sets
 $\{\omega : g_p(\omega,z)\leq a\}$ is measurable. As $\{p\in P^\Q(\C^n) : \deg(p)\geq 1\}$ is a countable set, we have  a countable intersection of measurable sets, which is again a measurable set. Thus $\Phi_{K(\cdot)}(z)$ is a measurable function.
 \end{proof}


\section{Approximation: from parameter-wise uniform to jointly uniform} \label{GeneralizedmeasurableFunctions}

Let $K$ be a random compact set in $\C^n$. In the theory of relations, which are the same as multifunctions, also called set-valued functions, it is standard to define and denote the graph of $K$ as 
\begin{equation*}
\Gr K := \{(\omega, z) \in \Omega\times \C^n : z \in K(\omega)\}.
\end{equation*}
If we consider a compact set $K\subset\C^n$ as a random compact set with $K(\omega)=K,$ for all $\omega\in K,$ then $\Gr K=\Omega\times K.$

It is also standard to set 
\begin{equation*}
K^{-1}(z) = \{\omega : z \in K(\omega)\} = \bigcap_{n=1}^\infty \{\omega : K(\omega) \cap B_{\frac{1}{n}}(z) \neq \emptyset\}.
\end{equation*}
Note that this does \emph{not} correspond to the pre-image of the compact set $\{z\}$. 
By the second equality, we see that this set is measurable.
We define a generalized random function as a function $f : \Gr K \rightarrow \C$ 
(hence $f(\omega,\cdot)$ is defined over $K(\omega)$), 
such that, for fixed $z\in \C^n$ with $K^{-1}(z) \neq \emptyset$, 
the function $f(\cdot,z)$ from $K^{-1}(z)$ to $\C$ is measurable. 
In case a compact set $K$ is considered as a random compact set, then a generalized random function on $\Gr K=\Omega\times K$ turns out to be just a random function on $K.$

We say a random compact set  $K$ is uniformly separable if there exists a countable subset $E$ of $\C^n$ whose intersection with 
$K(\omega)$ is dense in $K(\omega)$ for every $\omega \in \Omega$. In a way it is a generalization of a random compact set  taking countably many values. In fact, we have the following obvious proposition.

\begin{proposition}
Let $K$ be a random compact set. If
\begin{equation}\label{unifsep}
\abs{\bigcup_{\omega \in \Omega} K(\omega)\setminus \overline{\intr K(\omega)}} \leq \aleph_0,
\end{equation}
then $K$ is uniformly separable. Moreover, 
if $K(\omega)$ is the closure of a bounded open subset of $\C^n$ for all $\omega \in \Omega$, then $K$ is uniformly separable.
\end{proposition}

Note that the case that $K$ takes only countably many values is included in (\ref{unifsep}). We do not know whether (\ref{unifsep}) is a necessary condition when $K$ takes uncountably many values.

A generalized random  function $f(\omega,z)$ will be said to be a  generalized continuous function, if  for each $\omega$
the  function  $f(\omega, \cdot)$   is continuous on $K(\omega).$

\begin{theorem}\label{T:unifsep}
Let $K$ be a uniformly separable random compact set. Let ${f_j}$ be a sequence of  generalized continuous  functions  such that
\begin{equation*}
    f_j(\omega, \cdot) \longrightarrow f(\omega, \cdot) \text{~uniformly on } K(\omega), \quad 
		 \forall \omega \in \Omega .
\end{equation*}
Then there exists a sequence ${F_j}$ of generalized continuous  functions such that
\begin{equation*}
    F_j \longrightarrow f \text{~uniformly on } \Gr K,
\end{equation*}
and such that for all $\omega \in \Omega$ and $j \in \N$, there exists a $k \in \N$ with $F_j(\omega, \cdot) = f_k(\omega, \cdot)$.
\end{theorem}

\begin{proof}
The function $f$ is  well defined as a function from $\Gr K$ to $\C.$ For each $z\in\C^n$ such that $K^{-1}(z)\not=\emptyset,$ the function $f(\cdot,z)$ is easily seen to be measurable on $K^{-1}(z).$ Thus $f$ is a  generalized continuous function.

Denote by $E \subset \C^n$ the countable set that is dense in $K(\omega)$ for all $\omega \in \Omega$. Let $g$ be a generalized continuous function and let $\alpha \in \R$. By continuity of $g(\omega, \cdot)$, we have that
\begin{align*}
\{\omega &: \sup_{z \in K(\omega)} |g(\omega, z)| \leq \alpha\}\\ 
&= \bigcap_{z \in E}\left(\{\omega \in K^{-1}(z) : | g(\omega, z)| \leq \alpha\} \cup \{\omega \not\in K^{-1}(z)\}\right).
\end{align*}
It follows that the function $\omega \mapsto \sup_{z \in K(\omega)} |g(\omega,z)| = \|g(\omega, \cdot)\|_{K(\omega)}$ is measurable since the right-hand side is a countable intersection of measurable sets.

Let $\varepsilon > 0$ and consider the sets $A_n = \{\omega : \|f_k(\omega,\cdot) - f(\omega,\cdot)\|_{K(\omega)}\leq \varepsilon, \forall k \geq n\}$. By the previous discussion, these sets are measurable since $f_k - f$ is a generalized continuous function. Then $A_1 \subset A_2 \subset \dots$ and $\bigcup_n A_n = \Omega$ by uniform convergence on $K(\omega)$ for all $\omega \in \Omega$. Let $A_0 = \emptyset$ and $A_n^\prime = A_n \setminus A_{n-1}$. From these sets we consider the measurable function $\varphi: \Omega \rightarrow \N$ given by
\begin{equation*}
\varphi = \sum_{n} n \chi_{A_n^\prime}.
\end{equation*}

Define $F(\omega,z) = f_{\varphi(\omega)}(\omega, z)$. Let $U$ be a measurable subset of $\C$ and fix $z \in \C^n$. Then
\begin{align*}
\{\omega &\in K^{-1}(z): F(\omega, z) \in U\} \\
&= \bigcup_{k \in \N} \{\omega \in K^{-1}(z): f_k(\omega, z) \in U\} \cap \{\omega : \varphi(\omega) = k\}
\end{align*}
and therefore $F$ is a generalized random continuous function by the measurability of $f_k$ and $\varphi$. Also, by construction, we have that
\begin{equation*}
\abs{f(\omega, z) - F(\omega, z)} = \abs{f(\omega, z) - f_{\varphi(\omega)}(\omega, z)} \leq \varepsilon,
\end{equation*}
and so, by taking a sequence decreasing to $0$, we can construct a sequence $\{F_j\}$ with the desired properties.
\end{proof}

In particular, we can apply the previous theorem to families of functions.
For each nonempty compact set $K\subset\C^n,$ let $A(K)$ be a family of continuous functions on $K$. Thus,
$$
	A:\mathcal K^\prime(\C^n)\rightarrow \mathcal P(C(K)).
$$
Suppose further that if $Q \subset K$ is a compact set and $f \in A(K)$, then $f|_Q \in A(Q)$. For example, $A(K)$ can be the family of holomorphic, polynomial, or harmonic functions on $K$ 
or the family of rational functions having no poles on $K.$ 
For a random compact set  $K(\omega),$ we define a generalized $A(K)$-random function as  a generalized random
function $f : \Gr K \rightarrow \C$ such that $f(\omega, \cdot) \in A(K(\omega))$ for all $\omega \in \Omega.$

Suppose  $K:\Omega\rightarrow \mathcal K^\prime(\C^n)$ is a  
measurable compact subset of $\C^n,$ and  $K\mapsto A(K),$ for $K\in \mathcal K^\prime(\C^n),$ is as above.
Denote by $A_{[\Omega]}(K)$ the set of generalized $C(K)$-random functions $f : \Gr K \rightarrow \C,$ for which  there exists a sequence of generalized $A(K)$-random functions $\{f_j\}_{j=1}^\infty,$ 
such that, for each $\omega \in \Omega,$ 
$f_j(\omega, \cdot) \rightarrow f(\omega, \cdot)$ uniformly on $K(\omega).$
Denote by $A_{[\Omega]}^{\unif}(K)$ the set of generalized $C(K)$-random functions $f : \Gr K \rightarrow \C,$ for which there exists a sequence of generalized $A(K)$-random functions $\{f_j\}_{j=1}^\infty$ such that $f_j \rightarrow f$ uniformly on $\Gr K$. 

The following theorem follows directly from the previous one.

\begin{theorem} \label{convUnif}
Suppose  $K:\Omega\rightarrow \mathcal K^\prime(\C^n)$ is a uniformly separable random
compact subset of $\C^n$ and
$K\mapsto A(K),$ for $K\in \mathcal K^\prime(\C^n),$ is as above. 
Then $A_{[\Omega]}^{\unif}(K) = A_{[\Omega]}(K)$.
\end{theorem}

Suppose  $K\mapsto A(K),$ for $K\in \mathcal K^\prime(\C^n),$ is as above. 
Denote by $A_{\Omega}(K)$ the set of  $C(K)$-random functions $f : \Omega\times K \rightarrow \C,$ for which  there exists a sequence of $A(K)$-random functions $\{f_j\}_{j=1}^\infty,$ 
such that, for each $\omega \in \Omega,$ 
$f_j(\omega, \cdot) \rightarrow f(\omega, \cdot)$ uniformly on $K.$
Denote by $A_{\Omega}^{\unif}(K)$ the set of  $C(K)$-random functions $f : \Omega\times K \rightarrow \C$ for which there exists a sequence of  $A(K)$-random functions $\{f_j\}_{j=1}^\infty$ such that $f_j \rightarrow f$ uniformly on $\Omega\times K.$

\begin{corollary}\label{unif}
If $K\mapsto A(K),$ for $K\in \mathcal K^\prime(\C^n),$ is as above, then we have
$A_{\Omega}^{\unif}(K) = A_{\Omega}(K)$.
\end{corollary}

The following result came as a surprise to us, since $K$ is arbitrary and we do not assume that $f$ is defined in a neighborhood of $K.$ 

\begin{corollary}\label{surprise}
Let $K\subset\C^n$ be an arbitrary non-empty compact set,  and let
$f:\Omega\times K\rightarrow\C$ be a mapping. The following are equivalent. 

1. There is a sequence $r_j(\omega,z)$ of random rational functions, pole-free on  $K$, such that,  
$$
\mbox{for each} \,\, \omega, \quad r_j(\omega,\cdot)\rightarrow f(\omega,\cdot) \quad \mbox{uniformly on} \quad K.
$$

2. There is a sequence $r_j(\omega,z)$ of random rational functions, pole-free on  $K$, such that,  
$$
	r_j(\omega,z)\rightarrow f(\omega,z) \quad \mbox{uniformly on} 
		\quad \Omega\times K.
$$
\end{corollary}

\begin{proof}
The constant mapping $K:\Omega\rightarrow \mathcal K^\prime(\C^n),$ given by $\omega\mapsto K,$ is a uniformly separable random
compact subset of $\C^n.$ 
\end{proof}

The following particular case is worth a separate statement and is one of our main results. 

\begin{theorem}\label{weakConvUnif}
    $R_\Omega^{\unif}(K) = R_\Omega(K).$
\end{theorem}

Theorem \ref{conjRunge} combined with Theorem \ref{weakConvUnif}, yields 
the following rather strong  Runge  theorem.

\begin{theorem} \label{Runge}
 Let $f$ satisfy the hypotheses of Theorem \ref{conjRunge}. 
Then $f \in R_\Omega^{\unif}(K)$.
\end{theorem}

While the function classes $K\mapsto A(K)$ on compact sets considered in  Corollary \ref{unif} are quite general, covering, for example, many levels of smooth functions, it is desirable to have a similar result for function classes on open sets, in order  to include Hardy spaces and other important function classes defined on open sets. 
From Corollary \ref{unif}, we shall deduce an analogous corollary for approximation on compact subsets of an open set.  Suppose $U\mapsto A(U)$ for $U$ open in $\C^n,$ is a hereditary class of continuous functions, in the sense that, if $V$ is an open subset of $U$ and $f\in A(U),$ then $f|_V\in A(V).$ For a compact set $K\subset\C,$ let $A(K)$ consist of those continuous functions $f$ on $K$ for which there is an  open neighborhood $U$ of $K$ and $F\in A(U)$ such that $F|_K=f.$ Then $K\mapsto A(K)$ satisfies the hypotheses of Corollary \ref{unif}, so
from Corollary \ref{unif} we deduce the following. 

\begin{corollary}\label{locunif}
Suppose $U\mapsto A(U)$ is as above and $f\in A(U).$  If there exists an exhaustion $K_1\subset K_2\subset\cdots$ of $U$ and, for every $j$ and every $\omega$ a sequence $f_{k,j}$ of $A(K_k)$-random functions, such that $f_{k,j}(\omega,\cdot)\rightarrow f(\omega,\cdot)$ uniformly on $K_k,$ then there exists a sequence $f_k,$ where each $f_k$ is an $A(K_k)$-random function and 
$$
	|f_k(\omega,z)-f(\omega,z)|<1/k, \quad \forall \omega\in\Omega, \forall z\in K_k, \forall  k.
$$  

\end{corollary} 

\begin{proof}
Corollary \ref{locunif} follows by applying Corollary \ref{unif} to each $A(K_k)$ and taking a diagonal sequence. 
\end{proof}

\begin{corollary}\label{RungeCor 1}
Let $U$ be an open set in $\C$ and $f:\Omega\times U \rightarrow \C$ be a random holomorphic function. 
Let $K_1\subset K_2\subset\cdots$ be an exhaustion of $U$ by compact subsets and $\varepsilon_1, \varepsilon_2, \ldots$ be a sequence of positive numbers.
Then, there exists a sequence $R_1, R_2, \ldots,$ such that,   each $R_n$ is a random rational function, pole-free on  $K_n$ and   
$$
	|R_n(\omega,z)-f(\omega,z)| < \varepsilon_n, \quad \mbox{for all} \quad 
		(\omega,z)\in \Omega\times K_n. 
$$
\end{corollary}

\begin{proof} 
From Theorem \ref{Runge 1} and Theorem \ref{weakConvUnif}, for each $n,$ there is a sequence $R_{n,j}$ of random rational function, pole-free on  $K_n,$ such that $R_{n,j}\rightarrow f$ uniformly on $\Omega\times K_n.$ We conclude the proof by taking an appropriate diagonal sequence. 
\end{proof}

We shall now, following very closely the method of \cite{AB1984}, prove an ``almost everywhere'' version of a measurable Oka--Weil theorem. We first need the following lemmas.

\begin{lemma}\label{P closed}
	The set $P(k,m)$ of polynomials in $\C^n$ of degree at most $k$ and coefficients bounded by $m$ (in norm) is closed in every $C(K,\C)$ for every non-empty compact set $K\subseteq \C^n$.
\end{lemma}

\begin{proof}
	Fix a non-empty compact set $K$. Let there be a sequence $p_i$ of polynomials in $P(k,m)$ converging uniformly to a function $f\in C(K,\C)$. We wish to show $f$ is in the restriction of $P(k,m)$ over $K$.
	
	We denote $d_k$ the total number of possible terms in such polynomials. We in fact have that
	\begin{align*}
		d_k=\sum_{\ell=0}^k\binom{n+\ell-1}{n-1}.
	\end{align*}
	
	We denote $a_{i,\alpha}$, where $\alpha\in \N^n$ the coefficient of $z^\alpha$. We may then form the $d_k$-tuplets 
 of coefficients (by choosing an ordering) : $a_i\in \C^{d_k}$ for $p_i$. To simplify our work, we may endow $\C^{d_k}$ with the max norm :
$$\|a_i\|_\infty=\max_\alpha |a_{i,\alpha}|.$$
	
	The $a_{i,\alpha}$ are bounded in norm by $m$. Thus the $a_i$ are all in the ball $\overline{B_{m}}(0)\subset\C^{d_k} $, which is a compact set, and thus  the sequence $\{a_i\}_{i=0}^\infty$ has a subsequence $\{a_i'\}_{i=0}^\infty$ that converges to a point $b\in\C^{d_k}.$ We denote the associated subsequence of polynomials by $\{p_i'\}_{i=0}^\infty$. 

Since $p_i\rightarrow f,$ the same is true for the subsequence $p_i'.$ But $p_i'$ converges on all compact sets to the polynomial $p(z)=\sum_{|\alpha|\leq k} b_\alpha z^\alpha.$ Thus, $f=p|_K.$ 
\end{proof}

\begin{lemma}
	The set $P(C^n)$ of polynomials over $\C^n$ is a Suslin subset of $C(K,\C)$ for every non-empty compact set $K\subseteq \C^n$.
\end{lemma}
\begin{proof}
	Firstly, we have that
	\begin{align*}
		P(\C^n)=\bigcup_{m=0}^\infty\bigcup_{k=0}^\infty P(k,m).
	\end{align*}
	If we fix a compact set $K\subset \C^n$, then this is a countable union of sets each of which is closed relative to $C(K,\C)$, by Lemma \ref{P closed}. Thus $P(\C^n)|_K$ is a Borel subset of $C(K,\C)$.
	
	We have already shown in the proof of Proposition \ref{measurable element} that $C(K,\C)$ is separable. It is also complete, as a Cauchy sequence with the sup norm converges uniformly, and thus converges toward a continuous function. Thus $C(K,\C)$ is a Polish space.
	It is known  that every Borel subset of a Polish space is Suslin.
\end{proof}

\begin{theorem} \label{OkaWeil}
	Let $(\Omega,\mathcal A, \mu)$ be a $\sigma$-finite measure space. Let $K$ be a random compact set, whose range $\{K(\omega):\omega\in\Omega\}$ consists of at most a countable number of different compact sets. Suppose that $K$ is polynomially convex, i.e., $K(\omega)=\widehat K(\omega)$ for all~$\omega$. Let $f$ be a generalized random function such that $f(\omega,\cdot)$ is holomorphic in a neighborhood of $K(\omega)$ for each $\omega$, and let $\varepsilon$ be a positive measurable function defined on $\Omega.$ Then there exists $p$, a generalized random polynomial, such that$$\|p(\omega,\cdot)-f(\omega,\cdot)\|_{K(\omega)}<\varepsilon(\omega)$$
	for all $\omega$ outside a  measurable  set $L\subset \Omega$ such that $\mu(L)=0$.
\end{theorem}

\begin{proof}
We denote $\{K_j\}_{j=0}^\infty$ the values taken by $K$. We also denote 
	\begin{equation*}
	\Omega_j:=\{\omega \in \Omega : K(\omega)=K_j\}=K^{-1}(K_j).
	\end{equation*}
As this is the pre-image of a closed set (the single compact set), each $\Omega_j$ is measurable 
and $\Omega=\cup_j\Omega_j.$

	We shall define as measurable subsets of $\Omega_j$ the following $\sigma$-algebra : $\mathcal A_j:=\{P\cap \Omega_j : P\in \mathcal A\}$. But $\Omega_j$ is measurable in $\Omega$, and thus this definition is equivalent to $\mathcal A_j=\{P\subset \Omega_j : P\in \mathcal A\}$. Most importantly, every measurable set of $\Omega_j$ is measurable in $\Omega$.
	
	Let $f_j$ be the restriction of $f$ to $\Omega_j\times K_j.$ We claim this is a measurable function. Fix $z\in K_j$. We know that $f(\cdot, z)$ is a measurable mapping between $K^{-1}(z)$ and $\C$ and, for an open set $O$, $$(f_j(\cdot,z))^{-1}(O)=\Omega_j\cap (f(\cdot,z))^{-1}(O).$$ This proves that $f(\cdot,z)$ is a measurable function over $\Omega_j$, and thus $f$ is a random function.
	
	By Proposition \ref{measurable element}, we know that the mapping $\omega\mapsto f(\omega,\cdot)$ between $\Omega_j$ and $C(K_j,\C)$ is measurable.
	
	We know that $f_j(\omega,\cdot)$ is holomorphic in a neighborhood of $K_j$. But $K_j$ is, by hypothesis, polynomially convex. Thus, by the Oka--Weil Approximation Theorem,  the multivalued mapping
	\begin{align*}
		\psi_{j}(\omega):=\{q\in P(\C^n) : \|f_j(\omega,\cdot)-q(\cdot)\|_{K_j}<\varepsilon(\omega)\}
	\end{align*}
	is never empty.
	
	By Theorem 1 of \cite{AB1984}, and since $P(\C^n)$ is a Suslin subset of $C(K_j,\C)$, this means there exists a measurable selection $p_j : \Omega_j \rightarrow P(\C^n)$ which approximates $f_j$ to within $\varepsilon(\omega)$ almost everywhere, which means outside of a measurable set of measure $0.$
	
	We now define $p : \Omega\rightarrow P(\C^n)$ by setting $p|_{\Omega_j}=p_j$. Clearly, this function approximates $f$ almost everywhere. Indeed, if we denote $L_j$ the set of $\omega\in \Omega_j$ for which $p$ does not approximate $f$, we have that $L=\cup_{j=0}^\infty L_j$ and thus $\mu(L)=\sum_{i=0}^\infty \mu(L_j)=0$.
 Since the $L_j$ are measurable, $L$ is  measurable.

	It is only left to show that this is a generalized random function. By Proposition \ref{measurable element}, we have that for fixed $z\in K_j$, the mapping $\omega \mapsto p_j(\omega,z)$ is measurable. Thus, we can now say that, for a given open set $O\subset \C$ and $z \in \C^n$, the set
	\begin{align*}
		(p(\cdot,z))^{-1}(O)=\bigcup_{\{j\in \N: z\in K_j\}}(p_j(\cdot,z))^{-1}(O)
	\end{align*}
is a measurable set, as a countable union of measurable sets. 

Since this is a subset of the measurable set 
$$
	K^{-1}(z)=\bigcup_{j : z\in K_j}\Omega_j,
$$
the function $p$ is measurable on $K^{-1}(z)$ and hence $p$ is a generalized random function.
\end{proof}

It would be interesting to see what happens when trying to remove the ``almost everywhere'' or by reducing the hypothesis of the compact set. Could it be true if it were only uniformly separable, or even only measurable? 

The following corollary is the Oka--Weil Theorem, but since the Oka--Weil Theorem was invoked in the proof of Theorem \ref{OkaWeil}, the corollary should be considered more properly as a special case.

\begin{corollary}
Let $K\subset\C^n$ be a polynomially convex compact set and $f$ a measurable function in a neighborhood of $K,$ such that  $f(\omega,\cdot)$ is holomorphic in a neighborhood of $K(\omega)$ for each $\omega$, and let $\varepsilon$ be a positive measurable function defined on $\Omega.$ Then there exists $p$, a generalized random polynomial, such that
$$\|p(\omega,\cdot)-f(\omega,\cdot)\|_K<\varepsilon(\omega)$$
	for all $\omega$ outside a  measurable  set $L\subset \Omega$ such that $\mu(L)=0$.
\end{corollary}

 \begin{proof}
We apply the theorem to the constant random compact set 
$$K(\omega):=K, \quad \forall \omega.$$ 
\end{proof}

\begin{corollary}(Oka--Weil)
Let $K\subset\C^n$ be a polynomially convex compact set and $f$ be a  function holomorphic in a neighborhood of $K.$  Then, for each $\varepsilon>0,$ there exists a  polynomial $p$ such that$$\|p-f\|_K<\varepsilon.$$
\end{corollary}

\begin{proof}
Let  $(\Omega,\mathcal A, \mu)=([0,1],\mathcal B,\lambda)$ be the probability space $[0,1]$ with Lebes\-gue measure $\lambda$ on the Borel algebra $\mathcal B.$ Now, we apply the previous corollary to the constant measurable function $\varepsilon(\omega):=\varepsilon.$ Since $[0,1]$ has positive measure, there is some $\omega_0\in[0,1]$ such that 
$$\|p(\omega_0,\cdot)-f\|_K<\varepsilon$$
To conclude the proof, we set $p(z)=p(\omega_0,z).$ 
\end{proof}


\subsection*{Acknowledgements}
The first author was supported by grant RGPIN-2016-04107  from NSERC (Canada).
The second author was supported by grants from NSERC and the Canada Research Chairs program.


\end{document}